 \documentclass[12pt,reqno]{amsart}

\usepackage{amssymb}

\usepackage{color}

\usepackage{amssymb}
\usepackage{amsmath}
\usepackage{amsthm}
\usepackage{times}
\usepackage{graphicx}

\usepackage{color}
\usepackage{amsthm}
\newtheorem{theorem}{Theorem}

\newtheorem{corollary}{Corollary}
\newtheorem{proposition}{Proposition}

\usepackage[curve,all]{xy}
 \xyoption{curve}
 \xyoption{color}
 \xyoption{line}
\xyoption{arc}

\theoremstyle{definition}

\newtheorem{example}{Example}

\theoremstyle{remark}
\newtheorem{remark}{Remark}

\numberwithin{equation}{section}

\newcommand{\calA}{\mathcal{A}}

\newcommand{\calC}{\mathcal{C}}

\newcommand{\calE}{\mathcal{E}}
\newcommand{\calF}{\mathcal{F}}

\newcommand{\calL}{\mathcal{L}}
\newcommand{\calM}{\mathcal{M}}

\newcommand{\calO}{\mathcal{O}}

\newcommand{\El}{\mathcal{E}l}

\newcommand{\frakS}{\mathfrak{S}}

\newcommand{\bbF}{\mathbb{F}}
\newcommand{\bbC}{\mathbb{C}}

\newcommand{\bbP}{\mathbb{P}}

\newcommand{\bbZ}{\mathbb{Z}}

\newcommand{\ev}{\textup{even}}

\newcommand{\as}{\textup{as}}

\newcommand{\sfU}{\mathsf{U}}
\newcommand{\sfH}{\mathsf{H}}
\newcommand{\sfA}{\mathsf{A}}
\newcommand{\sfC}{\mathsf{C}}
\newcommand{\sfS}{\mathsf{S}}
\newcommand{\sfT}{\mathsf{T}}
\newcommand{\sfM}{\mathsf{M}}
\newcommand{\sfI}{\mathsf{I}}
\newcommand{\sfP}{\mathsf{P}}

\newcommand{\half}{\frac{1}{2}}

\newcommand{\SL}{\textrm{SL}}

\newcommand{\GL}{\textup{GL}}

\newcommand{\Hilb}{\textup{Hilb}}
\newcommand{\Hyp}{\textup{Hyp}}
\newcommand{\Cov}{\textup{Cov}}
\newcommand{\sfHyp}{\mathsf{Hyp}}

\newcommand{\PSL}{\textup{PSL}}

\newcommand{\Sym}{\textup{Sym}}

\newcommand{\PGL}{\textup{PGL}}

\newcommand{\iso}{\textup{iso}}

\newcommand{\He}{\textup{He}}

\newcommand{\Jac}{\textup{Jac}}

\newcommand{\co}{\textup{coarse}}

\newcommand{\da}{\dasharrow}
\newcommand{\bin}{\textup{bin}}
\newcommand{\Proj}{\textup{Proj~}}

\newcommand{\E}{\textup{El}}

\renewcommand\emptyset\varnothing

\newcommand{\beq}{\begin{equation}}
\newcommand{\eeq}{\end{equation}}

\begin{document}
\title[Rational self-maps of moduli spaces]{Rational self-maps of moduli spaces}

\author{Igor V. Dolgachev}
\address{Department of Mathematics, University of Michigan, 525 E. University Av., Ann Arbor, Mi, 49109, USA}
\email{idolga@umich.edu}


\begin{abstract} We discuss some examples of geometrically meaningful self-maps of  moduli space of curves of low genus and hypersurfaces.
\end{abstract}

\maketitle

\section{Introduction} Let $\calM$ be  a moduli problem of some algebraic geometrical objects.  We are interested in finding a geometrically meaningful dominant rational  self-map of $\calM$. To be more precise, we should consider a moduli functor $\calM$ on the category of schemes over  a field $\Bbbk$ that assigns to a scheme $S$ the birationally equivalence classes of families  $f:X\to S$ whose fibers over a dense open subset of $S$ are geometric objects which we wish to classify 
(e.g. smooth projective curves of fixed genus). The coarse moduli space $\sfM^{\co}$ of our moduli functor  $\calM$ is a scheme over $\Bbbk$ such that there exists a morphism from the functor $\calM$ to the Yoneda functor  $h_{\sfM^{\co}}$ in the category of schemes over $\Bbbk$ with rational maps as morphisms. It must be universal in a certain obvious sense and define a bijection $\calM(\Bbbk) \to \sfM^{\co}(\Bbbk)$. Clearly, the coarse moduli space, if it exists, is defined uniquely up to a birational isomorphism. A self-map of $\calM$ defines a rational self-map of $\sfM^{\co}$. More technically involved is the notion of a rational moduli stack and the problem of the existence of  its non-identical dominant rational self-map. Not being on the firm technical ground  in the theory of stacks, we are not pursuing this. All the problems we will be considering in this article deal with well-known moduli problems and  have a clear geometrical meaning. A reader who is not satisfied with the rigor  of the posed problem is welcome to  clarify it. 

 We will be concerned with the moduli problems for which the coarse moduli space exists. 
Our problem becomes to construct a natural self-transformation of the functor $\calM\to \calM$ that defines a dominant rational self-map of the coarse moduli space $\sfM^{\co}$. It seems that  examples of birational self-maps  are abundant (think about your favorite moduli spaces of objects with some additional level structure, for example the moduli spaces $\overline{\calM}_{g,n}$ of $n$-pointed stable curves of genus $g$) but the problem of finding dominant rational self-maps of degree $> 1$ is much harder.

In this article we will give two examples of rational self-maps of moduli spaces $\calM_g$ of curves of genus $g\le 3$, so far all attempts to find such maps for curves of genus $g > 3$ were unsuccessful.  
 
 For simplicity, we assume that the ground field is the field of complex numbers $\bbC$.

\section{Elliptic curves}

We start with two familiar examples.  Let $\El_n$ be the moduli problem defined by families of pairs $(C,L)$ consisting of 
an elliptic curve $C$ and a line bundle $L$ on $C$ of degree $n$. More precisely, $\El_n(S)$ consists of flat proper morphisms $p \colon X\to S$ of relative dimension 1 together with an invertible sheaf $\calL$ on $X$ satisfying the following properties:
\begin{itemize}
\item  $p$ is smooth over an open dense 
subset $U$ of $S$ with elliptic curves as fibres;
\item the restriction $\calL_s$ of $\calL$ to each fibre $X_s$ is an invertible sheaf of degree $n$.
\end{itemize}

We say that two families $X\to S$ and $X'\to S$ are equivalent if there exists an open dense subset $V$ of $S$ and  an isomorphism of $V$-schemes $\phi \colon X_V = X\times_SV\to X'\times_SV$ such that $\phi^*(\calL_V) \cong \calL_V'$. Here the subscript $V$ means that we restrict the sheaf over the preimage of $V$.  

Assume $n = 3$. The existence of a coarse moduli space for the moduli problem $\El_3$ follows from the following proposition.

\begin{proposition} Let $(p \colon X\to S,\calL)$ be a family of elliptic curves defined as above and let $\calE = p_*\calL$. Then $\calE$ is a locally free sheaf of rank 3 over $S$ and there exists a birational $S$-morphism $f \colon X\to W$, where $W$ is the zero scheme of a section 
$w\in\Gamma(\calO_{\bbP(\calE)}(3))$. 
\end{proposition}

\begin{proof} This is a modification of \cite[Proposition 1]{MS} on the existence of a Weierstrass form for families of elliptic curves with a section.  We use some standard properties of cohomology of a projective morphism (see \cite{Hartshorne}). Since $p$ is of relative dimension 1, for any invertible sheaf $\calF$ on $X$ of positive degree, the derived images $R^if_*\calF$ vanish for $i> 1$. The base change theorem allows us to compute the fiber of $R^1p_*\calF$ at closed points $t\in S$. We have 
$$R^1f_*\calF(t) \cong H^1(X_t,\calF\otimes \calO_{X_t}) = 0.$$  
Applying this to the sheaf $\calF = \calL^{\otimes n}$, we obtain that $R^1f_*\calL^{\otimes n}$ vanishes for $n > 0$ and 
\[
(f_*\calL^{\otimes n})(t) \cong H^0(X_t,\calL_t^{\otimes n}),
\]
where $\calL_t = i_t^*(\calL)$ and $i_t \colon X_t\hookrightarrow X$ is the closed embedding of the fiber. 

By  Riemann-Roch, the dimension of $H^0(X_t,\calL_t^{\otimes n})$ is equal to $3n$ (see \cite{MumfordL}). Thus the sheaves $\calE_n = f_*\calL^{\otimes n}$ are locally free of rank $3n$.

Taking $\calE = \calE_1$,  we have proved the first assertion.  Let us prove the remaining assertion. Let $U$ be an open affine subset 
over  which $\calE$ trivializes and 
let $x_U,y_U,z_U$ be a basis of the free $\calO(U)$-module $\calE(U)$. We find that all 10 monomials of degree 3 in $x_U,y_U,z_U$ belong to 
$\calE_3(U)$. Since $\calE_3$ is a free $\calO(U)$-module of rank 9, we get a linear relation between the monomials, hence there exists a 
nonzero cubic homogeneous form $F_U(X,Y,Z)\in \calO(U)[X,Y,Z]$ such that 
$F_U(x_U,y_U,z_U) = 0$. Taking an open cover $(U_i)_{i\in I}$ of $S$ trivializing $\calE$, we find that the restrictions of 
$F_{U_i}$ and $F_{U_j}$ to  $U_i\cap U_j$ are the same up to a projective linear transformation with coefficients in $\calO(U_i\cap U_j)$. 
Thus $(F_{U_i})_{i\in I}$ defines a section $w$ of the third symmetric power $S^3\calE$, or equivalently, a section of $\calO_{\bbP(\calE)}(3)$, where 
$\bbP(\calE)$ is the projective bundle associated with $\calE$, as defined by Grothendieck (see Hartshorne's book \cite{Hartshorne}).

By the property of adjoint functors, we have a canonical homomorphism of sheaves $p^*\calE \to \calL$ which defines an $S$-morphism $f \colon X\to \bbP(\calE)$ whose image is the subscheme $W$ of zeroes of $w$. Restricting $f$ to a smooth fibre $X_t$ we recognize a usual closed embedding of an elliptic curve into the projective plane $\bbP(\calE_t) \cong \bbP_{\kappa(t)}^2$, where $\kappa(t)$ is the residue field of $t$. It is given by the line bundle $\calL_t$ of degree $3$ on $X_t$. The image of the embedding is a cubic curve $F_t(X,Y,Z) = 0$, where $F_t(X,Y,Z)$ is obtained from 
$F_{U_i}(X,Y,Z)$, $t \in U_i$ by replacing the coefficients of $F_{U_i}$ with their images in the residue field $\kappa(t)$. This shows that $f$ is birational and a closed embedding over an open subset of $S$ over which the morphism is smooth. This proves the second assertion.
\end{proof}

We leave it to the reader to prove the following result.

\begin{corollary} The coarse moduli space $\E_{3}$ for $\El_3$ exists and is birationally isomorphic to the GIT-quotient of the projective 
space $\bbP(\bbC[X,Y,Z]_3)$ of homogeneous polynomials of degree 3 modulo the  linear group $\SL(3)$.
\end{corollary}
 
Recall from \cite{GT} that the GIT-quotient  is isomorphic to $\bbP^1$. The isomorphism is defined by 
the  Aronhold basic invariants  $\sfS$ and $\sfT$ of degrees $4$ and $6$ of the algebra of invariants $\Sym(V(3,3)^\vee)^{\SL(3)}$, so that $\E_3 = \Proj(\Sym(V(3,3)^\vee)^{\SL(3)}) \cong \Proj(\bbC[\sfS,\sfT]) \cong \bbP^1$. 

Now we can define an example of  a self-map of $\El_3$ by using the \emph{Hessian} of a cubic polynomial. Recall that the Hessian of a 
degree $d$ homogeneous form $P\in \bbC[T_0,\ldots,T_n]_d$ is the determinant $\He(P)$ of the matrix formed by the second partial derivatives of $P$. It is a homogeneous form of degree $(n+1)(d-2)$. The map $P\to \He(P)$ is an example of a \emph{covariant} of degree $n+1$ and order 
$(n+1)(d-2)$ on the space $V(n,d): = \bbC[T_0,\ldots,T_n]_d$.  In our case $\He(F)$ is a cubic ternary form. We define the self-map  
$\sfH \colon \El_3\to \El_3$ as follows. Given a family $(p \colon X\to S,\calL)\in \El_3(S)$, we choose a  trivializing 
open affine cover $(U_i)_{i\in I}$ of the locally free sheaf $p_*\calL$ that defines a  collection of cubic forms 
$\underline{F} = (F_{U_i}(X,Y,Z))_{i\in I}$ as above. We assign to $\underline{F}$ the collection of the 
Hessians $(\He(F_i))_{i\in I}$. By the covariance of the Hessian, they are glued together to define a section $\He(w)$ of $\calO_{\bbP(\calE)}(3)$ and hence 
a family $(X' = \He(w)\to S, \calO_{X'}(1))$ from $\El_3(S)$. Note, that the Hessian of a nonzero polynomial may be equal to zero, so 
it does not define a plane cubic curve. However, in our definition of the coarse moduli space, it is not a problem.

\begin{theorem} The degree of the self-map $\sfH \colon \E_{3}\da \E_{3}$ is equal to 3.
\end{theorem}

\begin{proof} Let  
\beq\label{hesse}
F(t_0,t_1;X,Y,Z)= t_0(X^3+Y^3+Z^3)+6t_1XYZ = 0.
\eeq
be the Hesse pencil of plane cubic curves (see \cite{CAG}, 3.1.3). Considered as a closed subvariety $X$ of $\bbP^1\times \bbP^2$, the 
first projection $\bbP^1\times \bbP^2\to \bbP^1$ restricted to $X$ defines a family $p \colon X\to \bbP^1$ of elliptic curves. It is smooth over 
the open subset $U = \bbP^1\setminus D$, where $D$ consists of four points $[0,1], [1,a], 1+8a^3 =0$. This gives a family from $\El_{3}(\bbP^1)$ which we call the \emph{Hesse family}.

It is known that any plane cubic curve is isomorphic to one of the members of the pencil.  One  computes the invariants $\sfS$ and $\sfT$ for a cubic 
curve in the Hesse form \eqref{hesse}.   We have 
$$\sfS = t_0^3t_1-t_1^4, \quad \sfT = t_0^6-20t_0^3t_1^3-8t_1^6.$$
 The Hesse family corresponds to the map   
$$f \colon \bbP^1\to \E_{3}\cong \bbP^1, \quad [t_0,t_1]\mapsto [(t_0^3t_1-t_1^4)^3,(t_0^6-20t_0t_1^3-8t_1^6)^2].$$
It is a Galois cover of degree 12  with the Galois group isomorphic to the alternating group $\mathfrak{A}_4\cong \PSL(2,\bbF_3)$ (see \cite{CAG}, 3.1.3). 

The explicit computation of the Hessian for the curve given by equation \eqref{hesse} 
gives  
$$\He(F(t_0,t_1;X,Y,Z)) = t_0t_1^2(X^3+Y^3+Z^3)-(t_0^3+2t_1^3)XYZ = 0.$$
This defines a degree 3 self-map of the base of the Hesse family, and hence a degree 3 self-map of $\E_{3}$. In coordinate-free approach, the Hessian covariant assigns to a plane cubic curve $C$ a pair $(\He(C),\epsilon)$, where $\epsilon $ is one of three non-trivial $2$-torsion divisor classes on $C$. The composition with the forgetting map  $(\He(C),\epsilon)\mapsto C$ is our self-map of degree 3.
\end{proof}

\begin{remark} The base of the Hessian family is naturally identified with the modular curve $X(3)$ representing the  fine
 moduli space $\bar{\sfA}_{1,3}$ of stable abelian curves with level $3$ structure. We have a natural rational transformation of the corresponding 
 moduli problems $\calA_{1,3}\to \El_3$. However, since there are families in $\El_3(S)$ which do not admit a section, 
 the map of functors is not surjective. The same construction defines a self-map of $\sfA_{1,3}$ of degree 3.
 \end{remark}

Next we assume that $n = 2$. We consider the families $p \colon X\to S$ of elliptic curves as above 
together with a line bundle $\calL$ of degree 2. We show, as above, that $\calE = p_*\calL$ is a rank 2 locally free sheaf  
and $p^*\calE \to \calL$ defines  $S$-morphism $f \colon X\to \bbP(\calE)$ of degree 2. The Stein factorization of $f$ is the 
composition of a birational $S$-morphism $\tau \colon X\to X'$ and a finite map of degree 2 $f' \colon X'\to \bbP(\calE)$ ramified over the 
zero subscheme of a section of $\calO_{\bbP(\calE)}(4)$. Locally, over an open affine subset $U$ of $S$, the family $X'$ is 
given by the equation $z_U^2+F_U(x_U,y_U) = 0$, where $x_U,y_U$ are local sections of $\calE$ and $z_U$ is a local section of 
$p_*\calL^{\otimes 2}$. The polynomial $F_U$ here is a homogeneous polynomial of degree 4 with coefficients in $\calO(U)$.

Let $\El_{2}$ be the moduli problem defined by the families from above. We leave to the reader the proof of the following.

\begin{proposition} The coarse moduli space $\E_2$ exists and is birationally isomorphic to the GIT-quotient 
of $\bbP(V(2,4))$ by the group $\SL(2).$
\end{proposition}

Recall that the GIT-quotient from above is isomorphic to $\bbP^1$. The isomorphism is defined by a free basis of the ring of 
invariants $\Sym(V(2,4)^\vee)^{\SL(2)}$ defined by invariants $\sfI_2$ and $\sfI_3$ of degree $2$ and $3$. 

To define a self-map of $\E_{2}$ we use again the Hessian covariant of binary quartics.

\begin{theorem} The degree of the self-map $\E_{2}\to \E_{2}$ defined by the Hessian covariant is equal to 2.
\end{theorem}

\begin{proof} We use the following canonical form of a binary quartic 
$$F(\alpha,X,Y) = X^4+6\alpha X^2Y^2+Y^4.$$
It is known that any binary form without multiple roots is equivalent under $\GL(2)$ to some $F(\alpha,X,Y)$. It is know that the algebra of $\SL(2)$-invariants of binary forms of degree 4 is generated by invariants $I_2$ and $I_3$ of degrees $2$ and $3$. The values of  
these invariants  on $F(\alpha,X,Y)$ are equal to $1+3\alpha^2$ and $\alpha-\alpha^3$, respectively. We have 
$$\He(F) = \alpha X^4+(1-4\alpha^2)X^2Y^2+\alpha Y^4.$$  Consider a curve $W$ in $Q= \bbP^1\times \bbP^1$ given by  the equation of 
bidegree $(1,4)$:
\beq\label{eq1}
t_0(X^4+Y^4)+6t_1X^2Y^2 = 0.
\eeq
Let $\pi \colon X\to Q$ be the  double cover of degree 2 ramified along the curve $W$. The composition of $\pi$ and the first projection 
$Q\to \bbP^1$ defines a family of elliptic curves over $\bbP^1$. Its fiber over a point $[a_0,a_1]$ is isomorphic to the 
double cover of $\bbP^1$ ramified over the zero scheme of the binary quartic $a_0(X^4+y^4)+6a_1X^2Y^2$. It is a smooth curve when the quartic has no multiple roots. 

Let $\bbP^1\to \E_{2}\cong \bbP^1$ be the map defined by $[t_0,t_1]\mapsto [(t_0^2+3t_1^2)^3, (t_0^2t_1-t_1^3)^2].$ 
This is a Galois cover of degree 6 with the Galois group isomorphic to $\frakS_3$. The Hessian covariant defines a self-map of 
$\bbP^1$ of degree 2 that descends to a degree 2 self-map of $\E_{2}$.
\end{proof}

\begin{remark}  A very interesting paper of Tom Fisher \cite{Fisher} introduces a generalization of a Hesse pencil of plane cubics to the case of elliptic normal  curves  of degree 4 and 5.  The author also extends to these cases the notion of the Hessian covariant. His calculations show that there exist self-maps of the moduli problems $\El_{4}$ and 
$\El_5$ of degrees $5$ and $11$, respectively. Both of the moduli problems admit  coarse moduli spaces based on  well-known models of elliptic normal curves of degrees 4 and 5.
\end{remark}

\section{Hypersurfaces} 

The moduli problem $\mathcal{H}yp_{d,n}$ of hypersurfaces of degree $d$ in $\bbP^n$ assigns to a scheme $S$ over a field $\Bbbk$ the set of  closed subschemes of $S\times \bbP^n$ that are flat over $S$ and whose fibers over points $s\in S$ are hypersurfaces of degree $d$ in $\bbP_{\Bbbk(s)}^n$. It admits a fine moduli space represented by a closed subscheme $\Hyp_{d,n}$ of the Hilbert scheme $\Hilb(\bbP^n)$.

Recall that  a \emph{covariant} of degree $r$ and order $m$ on the space $V(n,d)$ of homogeneous forms of degree $d$ in $n$ variables is 
 a $\SL(n)$-invariant map $C_{r,m} \colon V(n,d)\to V(n,m)$ given by  homogeneous polynomials of degree $r$ in the coefficients of a general 
polynomial in $V(n,d)$.
We will identify  a covariant with the corresponding rational $\PGL(n)$-equivariant map $$C_{r,m} \colon \bbP(V(n,d))\da \bbP(V(n,m)).$$ It defines a self-map of the moduli problem $\mathcal{H}yp_{d,n-1}$ in an obvious manner. A covariant is called \emph{dominant} if 
the rational map is dominant. We are interested in examples of dominant covariants. 

Similarly, one defines the notion of a \emph{contravariant} of degree $r$ and class $m$ on the space $V(n,d)$. It is a $\SL(n)$-equivariant polynomial map $V(n,d)\to V(n,m)^\vee$ given by  homogeneous polynomials of degree $r$ in the coefficients of a general 
polynomial in $V(n,d)$. If we view  $V(n,m)$ as the symmetric power $\Sym^d(E)$, where $E = \bbC^n$, then 
$V(n,m)^\vee = \Sym^d(E^\vee)^\vee = \Sym^d(E)$. By fixing a basis in $E$ and the dual basis in $E^\vee$, a covariant defines a rational self-map of the moduli problem $\mathcal{H}yp_{d,n-1}$. Again we are interested in dominant contravariants.

The group $\PGL(n)$ acts naturally on the functor $\mathcal{H}yp_{d,n-1}$ and the GIT-quotient $\Hyp_{d,n-1}/\!/\PGL(n)$ is the coarse moduli space of the moduli problem $\mathcal{H}yp_{d,n-1}/\PGL(n)$.

In this section we will be interested in the case $n = 2$. We denote by $\sfM_{\bin}^d$ the coarse moduli space for the moduli problem $\mathcal{H}yp_{d,1}/\PGL(2)$. We use covariants and contravariants to search for dominant self-maps of this moduli space. One of the standard ways to define a covariant uses the theory of transvectants. In modern form, this is the Clebsch-Gordan decomposition of the linear representations of $\SL(2)$:
\[
V(2,d)\otimes V(2,d')\cong \bigoplus_{k=0}^{[\half(d+d')]} V(2,d+d'-2k).
\]
One can give the projection map $\tau_k$ to the factor $V(2,d+d'-2k)$  by an explicit formula using the $k$th transvectant (see \cite{GT}, p. 68).

One constructs a new covariant  $C$ of degree $r+r'$ and order $m+m'-2k$ from known two covariants $C_{r,m}$ and $C_{r',m'}$ by 
considering the map 
$V(2,d)\to V(2,d)\otimes V(2,d'),\  B\mapsto C_{r,m}(B)\otimes C_{r',m'}(B)$ and composing it with the $k$-th transvectant $\tau_k$. For example, 
taking $k= 1$, we get the Jacobian determinant $B\mapsto J(C_{r,m}(B),C_{r',m'}(B))$. 

In particular, 
if $d = 2k$, we can take $m=m'=d$ and $C_{r,m} = C_{r',m'} = \sfU$, the identity covariant, and compose it with $\tau_k$ to obtain a covariant $V(2,d)\to V(2,d)$. If $d= 4$, it coincides with the Hessian covariant.

It is known that the  set of covariants of given order $m$ on $V(n,d)$ generate a finitely generated module $\Cov_{n,d}(m)$ over the ring of invariants $\Cov_{n,d}(0) = \Sym(V(n,d)^\vee)^{\SL(n)}$. Its generators are called the \emph{basic covariants}. It follows from their irreducibility that, when $m=d$ they define a non-identical rational map of the GIT-quotient $\bbP(V(n,d))/\!/\PGL(n)$.

We will be interested in covariants of order $m = d$.  Considering families from $\calM_{\bin}^d$ as maps $S\to (\bbP^1)^{(d)}$ and composing
these maps with the map $C_{r,d}$ we define a rational map $\sfM_{\bin}^d\da \sfM_{\bin}^d$. Provided the map is of finite degree, we get in this  way a self-map of the moduli space.
We have already considered an example of such a self-map for binary forms of degree 4. The next case is $d = 5$. 

In the case $d = 5$, the module $\Cov_{r,5}(5)$ is generated by three covariants $C_{1,5}, C_{3,5},$ and $C_{7,5}$ (see \cite{Elliott}, p. 235). The first covariant is the identity map $\sfU \colon V(2,5)\to V(2,5)$. The second one is $J(\sfU,C_{2,2})$, where  
$C_{2,2}$ is the $4$th transvectant. The third covariant is $J(\sfU,C_{6,2})$, where $C_{6,2}$ is a certain quadratic covariant of degree 6. 

We will be interested in the self-map defined by the covariant $C_{3,5}$. To compute it we use that a general binary quintic can be written in the \emph{Hammond form} \cite{Elliott}, p. 297:
$$B = at_0^5+bt_1^5+et_0^4t_1+ft_1^5.$$
This means that its orbit can be represented by a form with zero coefficients at {\small $t_0^3t_1^2, t_0^2t_1^3$. The value of $C_{3,5}$  at such a binary form is equal to
$$
C_{3,5}(B)=(af-5be)at_0^5+(5af-9be)bt_0^4t_1+8b^2ft_0^3t_1^2-8ae^2t_0^2t_1^3$$
$$+(5af-9be)et_0t_1^4-(af-5be)ft_1^5.$$
It defines a rational map $C:\bbP^3\da \bbP^5$ given by 
$[a,b,e,f]\mapsto [x_1,x_2,x_3,x_4,x_5,x_6]$, where $x_i$ is the coefficient of $C_{3,5}(B)$ at $t_0^it_1^{5-i}$. We claim that this map is of degree 1 onto its image. In fact, we have 
$ax_6 -fx_1 = 0,\quad bx_5-ex_2 = 0$ that shows that the pre-image of a general point $[x_1,x_2,x_3,x_4,x_5,x_6]$ is contained the  line $ax_6 -fx_1 = bx_5-ex_2 = 0 = 0.$ Restricting the map to this line, we find that it is of degree one onto a plane cubic. This shows that the pre-image of a general point consists of one point. Let 
$\Phi:\bbP^3\da \sfM_{\bin}^d$ be the rational map defined by the basic invariants $I_4,I_8,I_{12}$ of binary quintics (see \cite{GT}, \cite{Elliott})). Since  $C$ is given by a covariant, we have a commutative diagram
$$\xymatrix{\bbP^3\ar@{^{(}->}[r]\ar@{-->}[dr]&\bbP^5\ar@{-->}[r]^{C_{3,5}}\ar@{-->}[d]^\Phi&\bbP^5\ar@{-->}[d]^\Phi\\
&\sfM_{\bin}^5\ar@{-->}[r]^{\bar{C}_{3,5}}&\sfM_{\bin}^5}
$$
Let us show that the rational self-map $\bar{C}_{3,5}$ of the moduli space $\sfM_{\bin}^d$    is  of degree 1. In fact, suppose that the  pre-image of a general point $x$ in  $\sfM_{\bin}^d$ contains two distinct points. These two points  define two binary forms  $B,B'$ in $V(2,5)$ from two different orbits. They are sent under the map $C_{3,5}$ to two different binary forms $C_{3,5}(B), C_{3,5}(B')$  in $\bbP(V(2,5)) \cong \bbP^5$ with the same images under the map $\Phi$. This means  that they belong to the same orbit. Thus there exists $g\in \SL(2)$ such that 
$g(C_{3,5}(B)) = C_{3,5}(B')$. Since we are dealing with a covariant, we have $g(C_{3,5}(B)) = C_{3,5}(g(B))$. Thus $g(B)$ and $B'$ are in the same orbit, contradicting the assumption.

Let us record what we have found.

\begin{theorem} The coarse moduli space $\sfM_{\bin}^5$  of binary sextics admits a non-trivial birational self-map. 
\end{theorem}

It is well-known that the moduli space of  del Pezzo surfaces of degree 4 is isomorphic to the moduli space of binary quintics (use the Veronese map to assign to a binary quintic five points in the plane and blow them map to obtain a quartic del Pezzo surface). This gives

\begin{corollary} The coarse moduli space of quartic del Pezzo surfaces  admits a non-trivial birational self-map.
\end{corollary}

I do not know (yet?) any geometric meaning of this self-map.
 
\begin{remark} For any $f\in V(n,d)$ defining a nonsingular hypersurface in $\bbP^{n-1}$, the quotient of $\bbC[t_1,\ldots,t_n]$ by the Jacobian ideal $J(f)$ of $\bbC[t_1,\ldots,t_n]$ generated by partials of $f$ is a finite-dimensional vector space. Let $\He(f)$ be the Hessian form of $f$ of degree $N = n(d-2)$. The Jacobian algebra $\bbC[t_1,\ldots,t_n]/J(f)$ is a graded Gorenstein Artinian algebra with socle spanned by $\He(f)$. For any power 
$l^N$ of a linear form $l\in (\bbC^n)^\vee$, one can write
$$l^N = \as(f)(l)\He(f) \mod J(f),$$
for some homogeneous polynomial $\as(f)$ of degree $N$ on the dual linear space $(\bbC^{n})^\vee$. The function $f\mapsto \as(f)$ on the open subset of $V(n,d)$ of polynomials with non-vanishing discriminant $\Delta$ is called the \emph{associated form}.   Multiplied by some (smallest power) of $\Delta$ the  function $f\mapsto \as(f)$ becomes  a contravariant of order $n(d-2)$ on $V(n,d)$ and some degree. It was introduced and extensively studied by Alexander Isaev and his collaborators (see, for example, \cite{Alper0}).  

In two cases $(n,d) = (2,4),(3,3)$ the degree of the associated form $\as(f)$ coincides with the degree of $f$, and hence defines a rational self-map of the moduli space of binary quartics or cubic curves. It is proven that the degree of the self-map is equal to 1. This gives examples of a birational automorphisms of the moduli spaces $\sfM_{\bin}^4$ and $\E_1$. We refer to \cite{Alper} for some geometrical interpretation of these birational automorphisms.
\end{remark}

\section{Hyperelliptic curves} Recall that a hyperelliptic curve  is a smooth projective curve $C$ of genus $g > 1$ such that it admits a degree 2 finite morphism onto $\bbP^1$. The line bundle $\calL$ defining this map is unique and $\calL^{\otimes g-1}$ is isomorphic to the canonical 
sheaf $\omega_C$. 

A family of hyperelliptic curves is a flat proper morphism $f \colon X\to S$ such that its general fiber is a smooth curve of genus $g > 1$ and there exists a degree 2 finite surjective $S$-map $X\to \bbP(\calE)$, where $\calE$ is a rank 2 locally free sheaf. Two families are equivalent if they are isomorphic over a dense open subset of $S$. Let $\omega_{X/S}$ be the relative canonical sheaf. As in section 2, the sheaf 
$f_*\omega_{X/S}$ is locally free of rank $g$ and defines a morphism $\phi \colon X\to \bbP(f_*\omega_{X/S})$ whose image is 
$D = \bbP(\calE)$ (see \cite{Kleiman}, Theorem 5.5). The branch locus $W_{X/S}$ of $f \colon X\to D$ is a Cartier divisor on $D$ finite and flat  over $S$ 
of relative  degree $2g+2$ (it is called the Weierstrass subscheme in loc.\ cit). The correspondence $(X\to S)\to W_{X/S}$ defines an isomorphism of the moduli problems $\Hyp_g$ of hyperelliptic curves of genus $g$ and 
$\calM_{\bin}^{2g+2}$. 

So, we can define self-maps of $\Hyp_{g}$ by using covariants. However, not every self-map of $\Hyp_g$ or $\calM_{\bin}^{2g+2}$ arises from a covariant. 

\begin{example}  Fix two different points $x_0, x_1$ in the plane, a line $\calL$ through these points, and a nonsingular conic $Q$ tangent to the line $\calL$ at the point $x_1$. Consider the family of cuspidal cubics  $\calC$ with the cusp at $x_0$ and the cuspidal tangent equal to $\calL$.

\begin{figure}[h]
\includegraphics[scale=0.35]{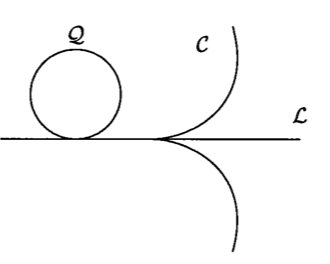}
\end{figure}
 
The cubic $\calC$ cuts out in $Q$ a divisor of degree 6. Identifying $Q$ with $\bbP^1$, we get a point in $\Hyp_2(\bbC)$. In fact, we can extend this construction to families. Now consider the double cover $X\to \bbP^2$ branched over conic $Q$, it is a nonsingular quadric. The preimage 
of $\calL$ is the union of two lines $\calL_1$ and  $\calL_2$ from different rulings intersecting at the preimage of the point $x_1$. The preimage of $\calC$ is a curve $W$ of bidegree $(3,3)$ on $X$ with two cusps $q_1$ and $q_2$, the preimages of $x_0$. Now take the double cover $S\to X$ branched over the 
union $W\cup \calL_1\cup \calL_2$. It is birationally isomorphic to a K3 surface. One can show that a minimal nonsingular model of $S$ admits a 
unique involution $\sigma \colon S\to S$ such that the quotient is a Kummer surface of a curve of genus 2 (see \cite{Appendix}). Also one proves that 
the corresponding self-map of $\sfHyp_2 = \sfM_{\bin}^6$ is of degree 16. An explicit formula for the self-map was found by A. Kumar \cite{Kumar}, p. 22. It shows  that the map is not defined by a covariant. 
\end{example}

One can consider certain natural finite covers of $\sfM_{\bin}^{2g+2}$ which are coarse moduli spaces of hyperelliptic curves with a level 
structure. For example, let $\sfM_{0,d}$ be the Knudsen-Mumford space of stable $d$-pointed curves of genus 0. The space $\sfM_{0,d}$ is birationally isomorphic to the GIT-quotient  $\sfP_1^d = (\bbP^1)^d/\!/\PGL(2)$. Forgetting about the ordering of the $d$ points defines a Galois map $\sfP_1^d\to \sfM_{\bin}^d$ with the Galois group isomorphic to $\frakS_d$. Choosing a subgroup $H$ of $\frakS_d$ we define a cover 
$$\sfM_{\bin}^{d,H}: = \sfP_1^d/H \to \sfM_{\bin}.$$

\begin{example} Let $d = 6$ and let $H\cong (\bbZ/2\bbZ)^3$ be the subgroup of the symmetric group $\frakS_6$ generated by transpositions $(12), (34)$ and $(56)$. 
For any hyperelliptic curve $C$ of genus 2, an order on the set of 6 Weierstrass points defines a full 2-level structure, i.e.\ an isomorphism 
of the vector space $\bbF_2^4$  with fixed symplectic structure and the group $H_1(C,\bbF_2)$ with the standard symplectic form defined by the 
cup-product (see \cite{CAG}, 5.2). We identify $H_1(C,\bbF_2)$ with the $2$-torsion subgroup $\Jac(C)[2]$ of the Jacobian variety $\Jac(C)$. The symplectic form is defined by the Weil-pairing. Thus the moduli space $P_1^6$ is birationally isomorphic to the moduli space $\Hyp_2(2)$ of hyperelliptic curves of genus 2 with a 2-level 
structure. Let $p_1,\ldots,p_6$ be the ramification points of $C\to \bbP^1$. The divisor classes $p_1-p_2, p_3-p_4,p_5-p_6$ define an isotropic plane in $\Jac(C)[2]$ together with a choice of a basis, i.e.\ a map $\bbF_2^2\to \Jac(C)[2]$ with image equal to an isotropic plane. The quotient $\Hyp_2(2)/H$ is the coarse moduli space $\Hyp_2(2)^{0}$ for families of pairs $(C,V)$, where $C$ is a genus 2 curve and $V$ is 
an isotropic plane with a fixed basis. If we enlarge $H$ considering the semi-direct product $H' = H\rtimes \frakS_3$, then the quotient $\sfM_{\bin}^{6,H'}$ becomes the coarse moduli space $\Hyp_2(2)^{\iso}$ of genus 2 curves together with a choice of an isotropic plane in $\Jac(C)[2]$.  The 
cover $\Hyp_2(2)^{\iso}\to \Hyp_2$ is of degree 15. 

The following construction of Friedrich Richelot defines a rational self-map $\sigma$ of degree 1 of $\Hyp_2(2)^{\iso}$. Put 3 unordered pairs 
$\{q_1,q_2\}, \{q_3,q_4\}, \{q_5,q_6\}$  on a nonsingular conic $C$ in the 
plane. They will be viewed as the Weierstrass points of $C$. Draw the lines  $\{\ell_1,\ell_2,\ell_3\}$ through each pair of points. The three lines form a triangle with vertices $\{a_1,a_2,a_3\}$ equal to the intersection points of three pairs of lines. Now draw the three pairs of tangents to $C$ from the points $a_1,a_2,a_3$. The tangency points define 6 unordered points  on $C$ which come with a choice of three unordered pairs. The construction is reversible. Starting with three pairs of points, we draw two tangents at each pair, their intersection points define three 
vertices of a triangle of lines  in the plane. The sides intersect $C$ at three pair of points. The  Richelot map is a generalization of the Gauss' 
arithmetic-geometric mean for curves of genus 1 (see \cite{AGM}). It assigns to a pair $(C,V)$ as above, the quotient $\Jac(C)/V$. It is a principally polarized abelian surface  which comes together with an isotropic plane $V^\perp/V$ in its group of 2-torsion points. One can show 
that is isomorphic to the Jacobian variety of a hyperelliptic curve $C'$ of genus 2.

\begin{figure}[h]
\includegraphics[scale=0.25]{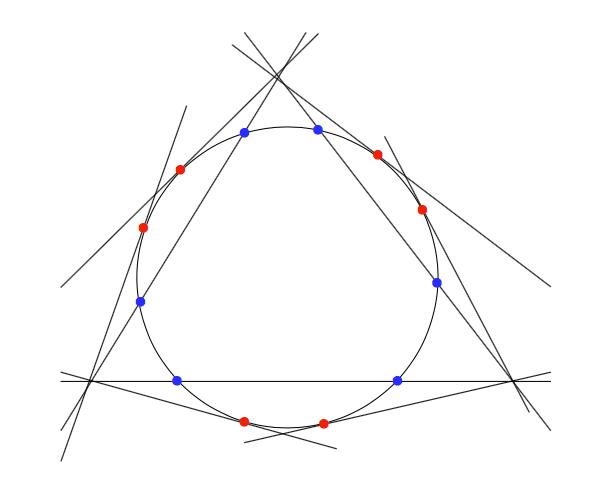}
\end{figure}

\end{example}

\begin{example} Let us consider the moduli space $\sfHyp_2(2)^0$ of ordered triples of pairs of points in $\bbP^1$. In the notation of the previous example, the lines $\ell_1 = \overline{p_1p_2},\ell_2 = \overline{p_3p_4},\ell_3 = \overline{p_5p_6} $ joining the pairs of points now come with an order. Also we can order the vertices of the corresponding triangle, say 
$a_1 = \ell_2\cap \ell_3, a_2 = \ell_1\cap \ell_3, a_3 = \ell_1\cap \ell_2$.  Now we define the following ordered set of 6 points $(q_1,\ldots,q_6)$. \begin{eqnarray*}
q_1=\  \text{the unique point on}\ \ell_1\  \text{such that}\  \{p_1,p_2\}, \{a_1,q_1\}\  \text{are harmonically conjugate},\\
q_2=\  \text{the unique point on}\ \ell_1\  \text{such that}\  \{p_1,p_2\}, \{a_3,q_2\}\  \text{are harmonically conjugate},\\
q_3 =\  \text{the unique point on}\ \ell_2\  \text{such that}\  \{p_3,p_4\}, \{a_3,q_3\}\  \text{are harmonically conjugate},\\
q_4 =\ \text{the unique point on}\ \ell_2 \ \text{such that}\  \{p_3,p_4\}, \{a_2,q_4\}\  \text{are harmonically conjugate},\\
q_5=\   \text{the unique point on}\ \ell_3\  \text{such that}\  \{p_5,p_6\}, \{a_2,q_5\}\  \text{are harmonically conjugate},\\
q_6= \   \text{the unique point on}\ \ell_3\  \text{such that}\  \{p_5,p_6\}, \{a_1,q_6\}\  \text{are harmonically conjugate}.\\
\end{eqnarray*}

\begin{figure}[h]
\includegraphics[scale=0.3]{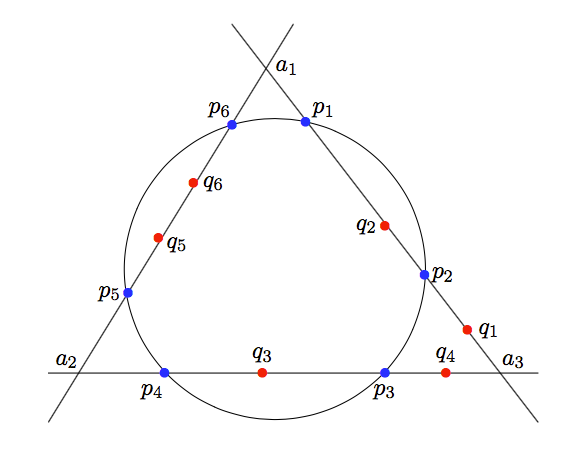}
\end{figure}

Recall that two pairs $\{a,b\}$ and $\{c,d\}$ of points on $\bbP^1$ considered as the zero schemes of binary forms 
$b_1 = \alpha t_0^2+2\beta t_0t_1+\gamma t_1^2, \ b_2 = \alpha' t_0^2+2\beta' t_0t_1+\gamma' t_1^2$ are called \emph{harmonically 
conjugate} if 
$$\alpha\gamma'-2\beta\beta'+\alpha'\gamma = 0.$$
We refer to \cite{CAG} for various equivalent definition of harmonic conjugacy. 
Next, we claim that the points $q_1,\ldots,q_6$ lie on a nonsingular conic. To check this, we may assume that the lines 
$\ell_1,\ell_2,\ell_3$ are the coordinate lines $x = 0, y = 0, z = 0$. Let
$$ax^2+by^2+cz^2+2dxy+2exz+2fyz = 0$$
be the equation of the conic $C$. The pair of points $\{p_1,p_{2}\}$ is on the line $x = 0$ and, in projective coordinates on this line, it is defined by a binary form $by^2+2fyz+cz^2$. The pair $\{a_1,q_1\}$ is defined by some binary form $b'y^2+2f'yz+c'z^2$ satisfying
$bc'+b'c-2ff' = 0$. One of the zeros of this binary form is the point $a_1 = [0,1,0]$. This gives $b' = 0$, $q_1 = [0,-c',2f']$,  
$bc'=2ff'$, hence $q_1 = [0,f,-b]$. Continuing in this way, we obtain
\[
\begin{array}{lll}
q_1 = [0,f,-b], & q_2 = [0,-c,f], & q_3 = [-c,0,e], \\[5pt]
q_4 = [e,0,-a], & q_5 = [d,-a,0], & q_6 = [-b,d,0].
\end{array}
\]
The six points lie on a conic if and only if the following identity holds:
\beq\label{coble}
(123)(145)(246)(356)-(124)(135)(236)(456)(123) = 0,
\eeq
where $(ijk)$ denotes the minor of the following matrix with columns $i,j,k$:
\[
\begin{pmatrix}0&0&-c&e&d&-b\\
f&-c&0&0&-a&d\\
-b&f&e&-a&0&0\end{pmatrix}.
\]
To explain this classical fact, one should fix the first 5 points and vary the sixth one and observe that the identity \eqref{coble} expresses a conic passing through the first 5 points. 

We compute
$$
(123) = -c(f^2-bc),\  (145) = a(be-df),\  (246) = def-abc,\  (356) = e(d^2-ab),$$
$$(124) = e(f^2-bc),\  (135) = def-abc, \  (236) = c(be-df),\  (456) = -a(d^2-ab).$$

Comparing the left-hand side with the right-hand side of \eqref{coble}, we find that they are equal.  

Now, we can define the self-map $\sigma \colon \sfHyp_2(2)^0\da \sfHyp_2(2)^0$ by  assigning to the projective equivalence class of the set $(\{p_1,p_2\}, \{p_3,p_4\},\{p_5,p_6\})$ the projective equivalence class of the set $(\{q_1,q_2\}, \{q_3,q_4\},\{q_5,q_6\})$. Here the projective equivalence for points on $\bbP^1$ is equivalent to the equivalence on $\bbP^2$ with respect to projective transformations leaving the conic invariant. Obviously, $\sigma$ is the composition of a map $\sfHyp_2(2)^0\da \sfHyp_2(2)$ and the forgetting map $\sfHyp_2(2)\da \sfHyp_2(2)^0$ 
of degree 8. Therefore, the degree of $\sigma$ is divisible by 8.

\end{example}

\begin{remark} In \cite{Mukai} S.\ Mukai constructs a biregular self-map of degree 8 of the Igusa compactification of the moduli space $\sfA_2(2)^0$ of principally polarized abelian surfaces $A$ together with a choice of an isotropic plane in $A[2]$. We believe that our map coincides with the Mukai map on the Jacobian locus.
\end{remark}

\section{Plane quartic curves}

Although the moduli problem $\calM_3$ for curves of genus 3 is different from the moduli problem $\mathcal{H}yp_{4,2}$ of plane quartics, their coarse moduli spaces $\sfM_3$ and the GIT-quotient $\Hyp_{4,2}/\PGL(3) = \bbP(V(3,4))/\!/\PGL(3)$ are birationally equivalent. Thus, we may view a self-map of $\sfM_3$ as a self-map of the latter moduli space.

There are many sources that discuss a construction of a birational map from the GIT-quotients $\sfM_3 = \bbP(V(3,4))/\!/\PGL(3)$ of the space of plane quartic curves to its cover $\sfM_3^{\ev}$ of degree 36 defined by a choice of an even theta characteristic (see, for example, \cite{CAG}, 6.3). The map is called the \emph{Scorza map}. Composing this map with the forgetful map $\sfM_3^{\ev}\to \sfM_3$, we obtain a rational self-map of $\sfM_3$ of degree 36.

The Scorza map is defined by the Clebsch covariant $\sfC_{4,4}$ of degree 4 and order 4 on the space $V(3,4)$ of ternary quartics. For any quartic 
$f\in V(3,4)$, the quartic $C_{4,4}(f)$ vanishes   on the locus of points such that the polar cubic with respect to this point lies in the zero locus of the Aronhold invariant $\sfS$. It is known that the $\SL(3)$-invariants of lowest degree on the space $V(3,4)$ are proportional to a cubic invariant $\sfI_3$. Multiplying it by the identity covariant $\sfU$, we get another covariant of degree 4 and order 4. This defines  a 2-dimensional vector space of covariants of degree 4 and order 4 and all covariants of degree 4 are obtained in this way (see \cite{Pascal}, Kapitel XVIII, \S 3).  

We denote by $C_{m,r}^*$ a contravariant on the space $V(n,d)$ of class $r$ and degree $m$ in the coefficients. The contravariant of lowest degree in coefficients of a ternary quartic form is the Clebsch contravariant $\Omega_{2,4}$ (see loc. cit and \cite{Salmon}, p. 264). It assigns to a ternary form $F$ of degree $4$ the ternary form of degree 4 in dual coordinates that vanishes on the set of lines that intersect the curve $F = 0$ at the union of two pairs of harmonically conjugate points. In other words, it means that the restriction of $F$ to the line is a binary form of degree $4$ on which the basic invariant  of degree $2$ vanishes. If we plug in the equation of a general line $L=\alpha x+\beta y +\gamma z = 0$ in the form $F(x,y,z)$ of degree $4$, we obtain a quartic binary form $F_L(x,y;\alpha,\beta,\gamma)$ of degree $4$ in $\alpha,\beta,\gamma$ and  of degree 1 in the coefficients of $F$. Evaluating the degree 2 invariant $I_2$ on this form we eliminate $x,y$ and obtain a quartic form in the dual coordinates $(\alpha,\beta,\gamma)$ of degree 2 in coefficients of $F$.

The following result based on ingenious computer computations was communicated to me by Damiano Testa.

\begin{theorem}[D. Testa] The degree of the rational self map  of $\sfM_3$ defined by the Salmon contravariant $\Omega_{2,4}$ is equal to 15.
\end{theorem}

Unfortunately, we do not know whether the coincidence of the number 15 with the number of non-trivial $2$-torsion divisor classes on a curve of genus 3 has anything to do with the construction of this self-map.

Another interesting contravariant of class $4$ and degree 5 is the \emph{Ciani contravariant} $\Omega_{5,4}$ (see the geometric interpretation of this contravariant in \cite{DK}, pp. 274--275). Together with the product $\sfI_3\Omega_{2,4}$, where $\sfI_3$ is an invariant of degree $3$, they span a 2-dimensional linear space of contravariants of degree 5 and class 4. 

\begin{remark} Suppose we have a one-parameter irreducible family $C_{r,m}(t), t\in T,$ of covariants  of the same degree and order. Then they define a rational map 
$$c:V(n,d)\times T\dasharrow V(n,d)\times T, \quad (B,t)\mapsto (C_{r,m}(t)(B),t).$$
Suppose $C_{r,m}(t_0)$ is dominant of some degree $k$. Since the source is irreducible, the image of the map $c$ is irreducible and contains an open subset in a hypersurface $V(n,d)\times \{t_0\}$ in the target space. Since the projection of the image to $T$ is dominant, this easily implies that the map $c$ is dominant. Since the dimension of fibers of the image over $T$ can only increase, we obtain that all covariants $C_{r,m}(t)$ are dominant.
The same argument applies to a family of contravariants.

Now suppose that we have an invariant $I_r$ of some degree $r$ and a covariant $C_{r+1,d}$ of degree $r+1$ and order $d$. Then 
$I_r\sfU$ is a covariant of degree $r+1$ and order $d$, obviously dominant. By the argument from above, we obtain that $C_{r+1,d}$ is dominant. 

 For example, we have a pencil of covariants of degree $4$ and order $4$ spanned by the Clebsch 
covariant $\sfC_{4,4}$ and the covariant $\sfI_3\sfU$, where $\sfI_3$ is a basic invariant of degree $3$. This gives a one-dimensional family of dominant self-maps of $\sfM_3$. Does any of them besides the  Clebsch covariant and the identity invariant have a geometric meaning? 

\end{remark}

\section{Acknowlegments} I thank  Maxim Fedorchuk and  Damiano Testa for interesting discussions on the subject of the article. 

\bibliographystyle{plain}

\end{document}